\documentclass[a4]{article}

\usepackage{amsmath,amsfonts,amsthm,amssymb,graphicx,tikz-cd,stmaryrd}
\usepackage{algorithmic}
\usepackage{algorithm}
\usepackage{fancybox}
\usepackage{appendix}

\theoremstyle{definition}
\newtheorem{thm}{Theorem}[section]
\newtheorem{deff}[thm]{Definition}
\newtheorem{lemm}[thm]{Lemma}
\newtheorem{prop}[thm]{Proposition}
\newtheorem{cor}[thm]{Corollary}
\newtheorem{rem}[thm]{Remark}

\newcommand{\zahl}{\mathbb{Z}}

\newcommand{\rat}{\mathbb{Q}}

\newcommand{\card}{\# }

\newcommand{\ratmap}{\dashrightarrow}

\newcommand{\pone}{\mathbb{P}^1}
\newcommand{\orbmult}{\lambda^\circ}
\newcommand{\ptmult}{\lambda^\bullet}
\newcommand{\proj}{\mathbb{P}}

\newcommand{\fct}[1]{\operatorname{#1}}

\newcommand{\per}{\fct{Per}}

\newcommand{\dyn}{\fct{Dyn}}
\newcommand{\ccor}{\fct{Corr}}

\newcommand{\slt}{\fct{SL}_2}
\newcommand{\pglt}{\fct{PGL}_2}

\title{Dynamical Systems of Correspondences on the Projective Line II: Degrees of Multiplier Maps}
\author{Rin Gotou}
\date{}
\author{Rin Gotou\thanks{Department of Mathematics, Graduate School of Osaka University \texttt{u661233h@ecs.osaka-u.ac.jp}}}

\begin{document}
\maketitle
\begin{abstract} This paper is a sequel of \cite{gotou2023dyn_corr_moduli}. In this paper, we consider a kind of inverse problem of multipliers. The problem is to count number of isospectral correspondences, correspondences which has the same combination of multipliers. We give a primitive explicit upper bound. In particular, for a generic rational map of degree $d$, there are at most $O(d^{10d})$ rational maps with the same combination of multipliers for the fixed points and the 3-periodic points. This paper also includes two proofs of a correction \cite{hutz-tepper2013errata}, which states that the multipliers of the fixed and 2-periodic points determines generic cubic morphism uniquely. One is done by proceeding the computation in \cite{hutz-tepper2013multiplier}. The other is done by more explicit computation with the help of invariant theory given in \cite{west2015moduli}.

\end{abstract}

\section{Introduction}
This paper is a sequel of \cite{gotou2023dyn_corr_moduli}. In \cite{gotou2023dyn_corr_moduli}, the author defined the moduli spaces $\dyn_{d,e}$ of dynamical systems of self-correspondences on the projective line $\pone$. The moduli space $\dyn_{d,e}$ parametrizes self-correspondences $C \subset \pone \times \pone$ of degree $(d,e)$ with only mild (i.e. of multiplicity $\leq \frac{d+e}{2}$) singularities on the diagonal of $\pone \times \pone$ (\cite[Theorem 1.1]{gotou2023dyn_corr_moduli}), up to conjugation by $\fct{Aut}(\pone ) = \pglt$. As essential structures of the moduli space of dynamical system, iteration maps $\Psi_n : \dyn_{d,e} \ratmap \dyn_{d^n,e^n}$ and fixed point multiplier map $\lambda_{1,(d,e)} : \dyn_{d,e} \ratmap \proj^{d+e}$ were introduced. These respectively indicates $n$-th iteration $C \mapsto C \circ \cdots \circ C$ and the fixed point multipliers
\[C : (f = 0) \mapsto \left\{ \left. \lambda_{z}(f) := - \frac{\partial_y f(z,z)}{\partial_x f(z,z)} \right| z\in \pone : f(z,z) = 0 \right\} \]
(under the isomorphism $\{ (d+e) \text{ (possibly multiple) points in }\pone \} = \fct{Sym}_{d+e} \pone \simeq \proj^{d+e}$).
Multiplier maps $\lambda_{n,(d,e)} := \lambda_{1,(d^n , e^n)} \circ \Psi_n$ were also introduced, but well-definedness of $\lambda_{n,(d,e)}$ for general $(n,d,e)$ are remaining as a problem (\cite[Problem 1.9, Remark 7.2]{gotou2023dyn_corr_moduli}). If the $n$-th multiplier map $\lambda_{n,(d,e)}$ is well-defined, then it indicates
\[ C : (f = 0 ) \mapsto \left\{ \left. \lambda_{(z_i )}(f) := \prod_{i = 0}^{n-1} \left( - \frac{\partial_y f(z_i,z_{i+1})}{\partial_x f(z_i,z_{i+1})} \right) \right| \genfrac{}{}{0pt}{}{ (z_i)_{i = 0}^{n} \in (\pone)^{n+1}: }{f(z_i ,z_{i+1} ) = 0, z_0 = z_n} \right\}. \]

A purpose to define morphisms $\lambda_n := \lambda_{n,(d,e)}$ is to consider inverse problem of multiplier, that is, how extant information of multipliers determines morphisms. The results about inverse problem of multiplier are rephrased as the properties of multiplier maps to their images. Let 
\begin{align*}
    \Lambda_n & := \prod_{m : m | n} \lambda_m : \dyn_{d,e} \ratmap \prod_{m : m | n} \proj^{d^m + e^m}  \text{ and } \\
    \Lambda & := \prod_{m \geq 1} \lambda_m : \dyn_{d,e} \ratmap \prod_{m \geq 1} \proj^{d^m + e^m},
\end{align*}
where $m | n$ means that $m$ divides $n$. Let $\fct{Poly}_d \subset \dyn_{1,d}$ be the locus of the points indicates a conjugation class including some polynomial morphisms. Table \ref{tab:multresult} is a brief review of known results about inverse problem of multipliers, as the form of degrees of multiplier maps to their images. All known results are about cases of usual morphisms, that is, the degree as self-correspondence is $(1,d)$.
\begin{table}[h]
    \centering
    \begin{tabular}{l|l|l|l|l}
     Reference & locus $X$ & degree of & map $F$ & degree of \\
     & & morphisms & & $F : X \ratmap F(X)$ \\
     \hline 
     \cite{mcmullen1993} & $\fct{Dyn}_{1,d}$ & $d \geq 2$ & $\Lambda$ & $< \infty$ \\
     \cite{gorbovickis2015rat_mult_alg_indep} &  &  & $\Lambda_n\ (n \geq 3)$ & $ < \infty$ \\
     \cite{schmitt2017cptf_stmaps_prep} & & & & $<$ Recursive formula\\
     \cite[Conjecture]{gorbovickis2015rat_mult_alg_indep} & & & $\Lambda_2$ & ? \\
     \cite{jixie2023multinj} & & $d \geq 4$ &  $\Lambda$  & 1 \\
     \cite{milnor1993lei},\cite{silverman1996p1moduli} & & $d = 2$ & $\Lambda_1 (= \lambda_1)$ & 1 \\ 
     (From dimension) & & $d \geq 3$ & $\Lambda_1$ & $\infty$ \\
     \cite{hutz-tepper2013multiplier} & & $d = 3$ & $\Lambda_2$ & $a_{3,2}$ \\
     \cite{fujimura2006polyn_moduli} & $\fct{Poly}_d$ & $d \geq 2$ & $\Lambda_1$ & $(d-2)!$ \\
     \cite{hutz-tepper2013multiplier} & & $d = 4,5$ &  $\Lambda_2$ & $1$ \\
    \end{tabular}
    \caption{ Results about degree of multiplier maps onto their images}
    \label{tab:multresult}
\end{table}

\begin{rem}
There are some more precise results about the degrees for the loci which the degrees of multiplier maps changes from generic behaviour (\cite{mcmullen1993}, \cite{silverman2007arithmetic}, \cite{fujimura2006polyn_moduli}, \cite{fujimura2007rat_moduli}, \cite{sugiyama2017polyn_fxd_pt_mult}, \cite{sugiyama2020moduli}). 
\end{rem}

In \cite{hutz-tepper2013multiplier}, $a_{3,2}= 12$ was stated, but the author corrected it to $a_{3,2} = 1$ (\cite{hutz-tepper2013errata}). An aim of this paper is to give precise proofs of correction.
\begin{thm}[Theorem \ref{cubbirat}]\label{maincubbir} The rational map  
\[ \Lambda_{2,(1,3)} = \lambda_{1,(1,3)} \times \lambda_{2,(1,3)} : \dyn_{1,3} \ratmap \proj^3 \times \proj^9
\] is birational to its image.
\end{thm}
In Section \ref{sc:cubbirat}, we give two proofs of this theorem. In Subsection \ref{ss_ffr}, we prove Theorem \ref{cubbirat} by proceeding the computation done in \cite{hutz-tepper2013multiplier}. This proof is the proof mentioned in \cite{hutz-tepper2013errata} and independent from other parts (except the programs in Subsection \ref{ss_ffr_pr} used for the proof) of this paper.

The other proof in Subsection \ref{ss_dc} is by a direct computation on the invariant ring given in \cite{west2015moduli}, which is the coordinate ring of $\dyn_{1,3}$. The computation is done by an interpolation method in Subsection \ref{ss_intp_alg}, and some unexpectedly simple relations among the coordinate functions of (Remark \ref{simprel}). A merit of this method is that the part without the direct computation of structure of the coordinate rings can be used for more general cases of $\Lambda_{n,(d,e)}$. The other aim of this paper is to give a primitive estimation of degree of multiplier map of correspondence along this method. This gives very rough as an upper bound, but gives a finite number.
\begin{thm}
Let $p$ be a prime number. If the $p$-th multiplier map is well-defined and
\[\Lambda_p := \lambda_{1,(d,e)} \times \lambda_{p,(d,e)} : \dyn_{d,e} \ratmap \Lambda_p (\dyn_{d,e}) \subset \proj (D_{d+e}) \times \proj (D_{M}) \]
is generically finite to its image, then its degree is at most \[
     \frac{\gcd(d+e,2) N^{de+d+e-3} (d+e-3)! (de-3)! }{2(d+e) \cdot (de+d+e-3)!},
\]
where
\[ N := 2(d+e-1) + \frac{2( (d^p-1)(d^p-d)-(e^p-1)(e^p-e))}{p(d-e)}. \]
\end{thm}
\begin{rem}
 We can give upper bounds in similar method for generic $\Lambda_n$ in the similar assumption, but $N$ becomes slightly more complicated polynomial. We also note that $\Lambda_2$ can be generically finite to its image only if $(d-e)^2 \geq d+e-2$ (Remark \ref{dimsnd} and Remark \ref{immults}).
\end{rem}
Combining with the finiteness result (\cite{gorbovickis2015rat_mult_alg_indep} in Table \ref{tab:multresult}), we can see the following:
\begin{cor}
    For $d \geq 2$, the degree of $\Lambda_{3,(1,d)}$ is at most 
    \[ \frac{\gcd(d+1,2) 2^{2d-2}(d^5+d^4-d^2+2d)^{2d-2} \cdot (d-2)!(d-3)! }{2 \cdot 3^{2d-2} \cdot (d+1) \cdot (2d-1)!}. \]
\end{cor}
For $d = 3$, this only gives the evaluation $ \deg \Lambda_{3,(1,d)} \leq 4369320$.
\begin{rem}
In \cite{schmitt2017cptf_stmaps_prep}, an algorithm to count the degree of multiplier maps using equivariant Gromov-Witten invariant is given. The author have not completed the evaluation of the order of the recursion formula.
\end{rem}
This paper is organized as follows. In Section 2, we set up notation and terminology. In Section 3, we give an evaluation of extension degree of rational function field by using the Hilbert series. In Section 4, we evaluate the degree of multiplier map using the evaluation in Section 3. In Section 5, we give two proofs of Theorem \ref{maincubbir}. 

\subsection*{Acknowledgments}
I express my deepest appreciation to Seidai Yasuda for providing MAGMA\cite{magma1997} emvironment and proofreading the original version of this paper, which is the latter half of my master thesis consists of Section \ref{sc_volalg} and the $\dyn_{1,d}$ case of Section \ref{sc_degbd}. I am also grateful to Takehiko Yasuda for supervising me and giving useful discussion. I apologize to Benjamin Hutz, Zhuchao Ji and Junyi Xie for the late uploading of the pricise of the correction. I also appliciate to Zhuchao Ji and Junyi Xie for contacted me about \cite{hutz-tepper2013errata}.

The author is granted by Kakenhi DC1(22KJ2090(21J22197)).

\section{Notation and Terminology}

Throughout this paper, we refer \cite{liu2006alg_geom} for the terminology of algebraic geometry.

We fix a field $k$ of characteristic zero. Unless otherwise stated, we consider any scheme as a scheme over $k$. 

For a ring $R$ and a free $R$-module $M$ of finite rank, we denote by $R[M]$ the polynomial ring generated by a basis of $M$ (with suitable identifications between different choices of basis). 
If a group $G$ and a representation $\rho : G \to \fct{Aut}_R (M)$ are also given, we write $R[M]^G$ for the invariant ring.

\section{Volume of Algebra and Rational Field of Projective Variety}\label{sc_volalg}
To evaluate degrees of multiplier maps, we need to evaluate the degree of a rational map to its image. In \cite{gotou2023dyn_corr_moduli}, the moduli space $\dyn_{d,e}$ of dynamiacal systems is given by the projective scheme $\fct{Proj} I(V_d \otimes V_e )$ of the naturally graded invariant ring $I(V_d \otimes V_e) := k[V_d \otimes V_e]^{\slt}$. Here we have a problem that the graded ring is not fully generated by linear terms. Moreover, full generator (secondary invariants) and relations (syzygies) are not known for generic cases (\cite{olive2017gordan}). Moreover, we only have little information about multiplier maps. So we use an evaluation only using Hilbert series. We use a trivial evaluation (Proposition \ref{Voldeg}), maybe well-known for experts.

\begin{deff}(Hilbert-Poincar\'{e} Series) For a graded $k$-algebra $A = \bigoplus_{i=0}^{\infty} A_i$, {\em the hilbert series of }$A$ is the series\[ H_A(t) := \sum_{i=0}^{\infty} \left( \dim_k A_i \right)t^n . \]
\end{deff}
\begin{deff} For a graded $k$-algebra $A = \bigoplus_{i=0}^{\infty} A_i$ of Krull dimension $d$, {\em the volume of }$A$ is defined by
\[ \fct{Vol}(A) :=  \lim_{t \to 1} (1-t)^{d} H_A(t) . \]
\end{deff}
\begin{rem} In \cite{derksen-kemper2015comp_i_t}, the degree of $A$ is used instead of the volume of $A$. We choose the word ``the volume of $A$'' to avoid confusing with the extension degree of algebras.  
\end{rem}
\begin{deff} Let $A = \bigoplus_{i=0}^{\infty} A_i$ be a graded $k$-algebra.
\begin{enumerate} \item the algebra $A$ is {\em saturated} if $\dim A_i \neq 0$ for any sufficiently large $i$.
\item {\em The saturator} of $A$ is the minimal positive integer $n$ such that the algebra $A^{[n]} := \bigoplus_{i=0}^{\infty }A_{ni}$ is satulated by the grading which $A_{ni}$ is degree $i$.
\end{enumerate}
\end{deff}

\begin{prop}\label{dimvol} Let $A = \bigoplus_{i=0}^{\infty} A_i $ be a finitely generated saturated graded $k$-algebra of Krull dimension $d$ which is an integral domain. Then we have
\[ \dim_k A_i = \frac{\fct{Vol}(A)i^{r-1} }{(r-1)!} + O(i^{r-2}) \ (i \to \infty ). \]
\end{prop}
\begin{proof}
From the proof of \cite[Proposition 1.4.5]{derksen-kemper2015comp_i_t}, it follows that if $B$ is a finitely generated graded algebra, the Hilbert series $H_{B}(t)$ has the form
\begin{equation} H_B(t) = \frac{P_B(t)}{\prod_{i=1}^d (1-t^{a_i})} \label{hilbB} \end{equation}
where $P_B(t)$ is a polynomial, $d$ is the Krull dimension of $B$ and $a_i$'s are positive integers.

Let us take nonzero elements $f_m \in A_m$ and $f_n \in A_n$ for coprime positive integers $m$ and $n$. We assumed that $A$ is an integral domain, therefore $f_n$ and $f_m$ are both regular element of $A$.
We have 
\begin{equation} H_A(t) = (1-t^m) H_{A/f_mA}(t) = (1-t^n) H_{A/f_nA}(t) \label{fmfn} \end{equation} for the natural gradings on $A/f_mA$ and $A/f_nA$. By writing (\ref{fmfn}) in the form of (\ref{hilbB}), we obtain that the order of the pole of $H_A (t)$ other than $1$ is at most $d-1$. Thus the Hilbert function $H_A(t)$ has the partial fraction decomposition of the form 
\[ H_A(t) = \frac{\fct{Vol}(A)}{(1-t)^{d}} + \sum_{i=1}^{d-1} \sum_{\zeta \in \mu_{\infty } } \frac{n_{\zeta, i} }{(1-\zeta t)^{i}} , \]
where $\mu_{\infty}$ is the set of all roots of unity and $n_{\zeta , i} = 0$ except for finitely many $\zeta \in \mu_{\infty}$.
The Taylor expansion 
\[ \frac{1}{(1-t)^e} = \sum_{i = 0}^{\infty} \binom{e+i-1}{e-1} t^n  \]
leads to the assertion of proposition.
\end{proof}
\begin{cor}\label{dvcor}
Let $A = \bigoplus_{i=0}^{\infty} A_i $ be a finitely generated graded $k$-algebra of Krull dimension $d$ which is an integral domain. Let $n$ be the saturator of $A$.
Then we have
\[ \dim_k A_{ni} = \frac{n \fct{Vol}(A)(ni)^{r-1} }{(r-1)!} + O((ni)^{r-2}) \ (i \to \infty ). \]
\end{cor}
\begin{proof}
Let $H_{A^{[n]}}(t) = \frac{\fct{Vol}(A^{[n]}) }{(1-t)^d} + \sum_{i=1}^{d-1} \sum_{\zeta \in \mu_{\infty } } \frac{n_{\zeta, i} }{(1-\zeta t)^{i}}$ be the partial fraction decomposition. Here we have $H_{A}(t) = H_{A^{[n]}}(t^n)$, therefore
\begin{align*} \fct{Vol}(A) &=  \lim_{t \to 1} (1-t)^d \left( \frac{\fct{Vol}(A^{[n]}) }{(1-t^n )^d} + \sum_{i=1}^{d-1} \sum_{\zeta \in \mu_{\infty } } \frac{n_{\zeta, i} }{(1-\zeta t)^{i}} \right) \\
&= \fct{Vol}(A^{[n]}) \cdot \lim_{t \to 1} \frac{(1-t)^d}{(1-t^n)^d} \\
&= \fct{Vol}(A^{[n]}) \cdot \frac{1}{n^d}.
\end{align*} 
By substituting this into Proposition \ref{dimvol} for $A^{[n]}$, we obtain the assertion.
\end{proof}
From now on, we fix a graded algebra $A := \bigoplus_{i=0}^{\infty} A_i$ which is an integral domain. We also fix a graded subalgebra $B$ of $A$. For any graded algebra $C$, we write $KP(C)$ for the rational function field $K(\fct{Proj}C)$. We have 
\[ KP( A) = \bigcup_{i = 0}^{\infty} \left\{  \left. \frac{a_i}{a'_i} \ \right| a_i \in A_i,\ a'_i \in A_i \setminus \{ 0 \} \right\}. \]
We assume that $KP(A)$ is a finite extension of $KP(B)$ and write $D$ for the degree of extension.
\begin{prop}There exists a $KP(B)$-basis of $KP(A)$ of the following form:
\[ \left\{ \left. \frac{a_i}{b_i} \ \right| i = 1, 2,\ldots ,D,\ b_i \in B_{n_i} , a_i \in A_{n_i} \right\}. \]
\end{prop}
\begin{proof}We write $K$ for the field $KP(B)$. Let $ \{ \frac{a_i}{a'_i} \mid i = 1,2,\ldots ,D,\  a_i,a'_i \in A_{n_i }\} $ be a $K$-basis of $KP(A)$. If $B_{n_i} = 0$ we replace $(a_i, a'_i)$ by $( aa_i , aa'_i )$ for $a \in A$ of sufficiently large degree, and then we can assume $B_{n_i} \neq 0$ and take $b_i \in A_{n_i} \setminus \{ 0 \}$. We have \[ K\left( \frac{a_1}{b_1} , \frac{a'_1}{b_1}, \ldots ,\frac{a_D}{b_D}, \frac{a'_D}{b_D} \right) = KP(A) \]
and each $\frac{a_i}{b_i} $ or $\frac{a'_i}{b_i}$ is integral over $K$. Therefore, for sufficiently large $N$, 
\[ \left\{ \left.  \frac{a_1^{e_1} \cdots a_D^{e_D}  a_1^{\prime e'_1} \cdots a_D^{\prime e'_D}}{b_1^{e_1+e'_1} \cdots b_D^{e_D+e'_D}} \ \right| \sum_i e_i + \sum_i e'_i \leq N \right\} \]
is a generator of $KP(A)$ as a $K$-vector space.
\end{proof}
By reducing to a common denominator, we obtain the following: 
\begin{cor}\label{abbasis}There exists a $KP(B)$-basis of $KP(A)$ which is the form
\[ \left\{  \left. \frac{a_i}{b_0} \ \right| i = 1, 2,\ldots ,D,\ b_0 \in B_{n} , a_i \in A_{n} \right\}. \]
\end{cor}

\begin{prop}\label{Voldeg} We have
\[ \frac{ \fct{Vol}(A)}{ \fct{Vol}(B) } \geq \frac{ s_A \fct{Vol}(A)}{s_B \fct{Vol}(B) } \geq [KP(A):KP(B)], \]
where $s_A$ and $s_B$ are the satulators of $A$ and $B$ respectively.
\end{prop}
\begin{proof}
The first inequality is immediate from the inclusion $B \subset A$. 

Let $ \{ \frac{a_i}{b_0} \mid i = 1, 2,\ldots ,D,\ b_0 \in B_{m} , a_i \in A_{m} \} $ be a $KP(B)$-basis of $KP(A)$ given by Corollary \ref{abbasis}. Then, the morphism \[ B^{\oplus D} \ni (b_i) \mapsto \sum_{i = 1}^{D} a_ib_i \in A \]
is injective. Therefore, we have $\dim A_{m+n}  \geq D\dim B_m $ for an arbitrary $m$. By applying Corollary \ref{dvcor}, we have \[ \frac{s_A \fct{Vol}(A)(m+n)^{r-1} }{(r-1)!} + O((m+n)^{r-2}) \geq D \frac{s_B \fct{Vol}(B)m^{r-1} }{(r-1)!} + O(m^{r-2}) (m \to \infty ). \]
Therefore, we have
\[  s_A \fct{Vol}(A)  \geq D s_B \fct{Vol}(B). \]
\end{proof}





\section{Degree Bound of Multiplier Map}\label{sc_degbd}

The volumes of the invariant algebras of irreducible representations of $\slt$ are classically calculated by Hilbert and the reducible cases are done in \cite{dc-h-h-s2020hilbert}.
\subsection{Schur Polynomial}
In \cite{dc-h-h-s2020hilbert}, Schur Polynomials are used to express the volumes of invariant rings. We briefly introduce the polynomials in a form that we can instantly give an explicit evaluation of the volumes.

\begin{deff}(Schur Polynomial) For a sequence of nonnegative integers $(d_i)$ of length $l$, the {\em Schur polynomial} $s_{(d_i)}(x_i)$ of $(d_i)$ is the symmetric polynomial of $l$ variables such that
\[ s_{(d_i)}(x_i) = \frac{\det (x_i^{j + d_j-1 })_{i,j = 1}^l }{\det (x_i^{j-1})_{i,j=1}^l} . \]
\end{deff}
\begin{deff} For a nonincreasing sequence of nonnegative integers $(d_i)$, the Young tableau of $(d_i)$ is the set of lattice points
\[ T(d_i) := \{ (i,j) \in \zahl^2 \mid 1 \leq i \leq d_j \}. \]
The set of $n$-semistandard tableau is the set 
\[ \fct{SST}_n(d_i) :=  \{ f: T(d_i) \to \{ 1, \ldots , n \} \mid f(i ,j ) \leq f(i+1 , j) , f(i, j)<f(i,j+1)\}. \]
\end{deff}
\begin{thm} \label{kostka}(Kostka's Definition \cite[Corollary 12.5]{prasad2019schur}) For a decreasing sequence of nonnegative integers $(d_1 , \ldots , d_l)$, we have
\[ s_{(d_1, \ldots , d_l , 0, 0, \ldots, 0 )} (x_1, \ldots, x_{l+k}) = \sum_{f \in \fct{SST}_{l+k}(d_i)} \prod_{(i,j) \in T(d_i)}x_{f(i,j)} . \]
\end{thm}
\subsection{Degrees of Linear Systems of Multiplier Maps}
The graph cycle $\Gamma_{d,e}$ over the moduli of correspondence is the irreducible hyperplane
\[ V_+ \left( \sum_{i,j} a_{ij}x^iy^j \right) \subset \proj (\langle a_{ij} \rangle ) \times \pone_x \times \pone_y = \ccor_{d,e} \times \pone_x \times \pone_y. \]
As a hyperplane, this is of degree $(1,d,e)$. Let $f(x,y) := \sum_{i,j} a_{ij}x^iy^j$ be the defining polynomial of $\Gamma$. The $n$-th iteration of $f$, $\Psi_n (f)$ is given by
\[ \Psi_n (f) (z_0,z_n) = \fct{res}_{z_1 ,\ldots , z_{n-1} }(f(z_0,z_1), f(z_1, z_2) ,\ldots , f(z_{n-1},z_n)). \]
Let $\Psi_n \Gamma$ be the graph of iteration morphism
\[ \Psi_n \Gamma := V_+ \left( \Psi_n(f)(x,y) \right) \subset \ccor_{d,e} \times \pone_x \times \pone_y. \]
By using the expression of composition using the resultant \cite[Section 5]{gotou2023dyn_corr_moduli} and the Sylvester formula, $\Psi_n \Gamma$ is of degree $(\frac{d^n-e^n}{d-e} , d^n, e^n)$.
The cycle of periodic points $\per_n$ of period $n$ is given by 
\[ \per_n := V_+ ( \Psi_n (f)(z,z) ) \subset \ccor_{d,e} \times \pone_z.\]
From the degree of $\Psi_n \Gamma$, the degree of $\per_n$ is $(\frac{d^n-e^n}{d-e} , d^n + e^n)$.

We remark that the cycle of periodic points $\fct{Per}_n$ includes the cycles of fixed points, and moreover the cycles $\fct{Per}_m$ for $m|n$. We define the scheme $\per_n^*$ of periodic points of formal period $n$ by extracting the periodic points of shorter periods. More explicitly, we define effective divisors $\per^*_n$ inductively as
\[ \per^*_1 := \per_1,\per_n^* := \per_n - \sum_{m < n, m|n} \per_m. \]

Let $\nu_n(x)$ be the family of polynomials, asymptotically defined by 
\[ \nu_1(x) = x, \nu_n(x) = x^n - \sum_{m < n, m|n} \nu_m(x). \]
In a closed form, $\nu_n$ is written by using the M\"{o}bius function $\mu$,
\[ \nu_n(x) = \sum_{m|n} \mu\left( n/m \right) x^n. \]
Then the degree of $\per^*_n$ is given by 
\[ \left( \frac{\nu_n(d) - \nu_n (e)}{d - e}, \nu_n (d) + \nu_n (e) \right) . \]
We write $\Pi_n^* f(z)$ for a defining form of the divisor $\per^*_n$.
\begin{rem} We have $\nu_n(1) = 0$ for $n > 1$. For the cases only considering rational maps, $\nu_n(d) + \nu_n(1)$ is sometimes used instead of $\nu_n(d)$ (for instance, $\nu_n(d)$ in \cite[Remark 4.3]{silverman2007arithmetic} and $N_n(d)$ in \cite[Chapter 4]{dolotin2006universal}).
\end{rem}

\begin{prop}\label{linsys}
    If the multiplier map
    \[ \lambda_{n , (d,e) } : \ccor_{d,e} \ratmap \ccor_{d^n , e^n} \ratmap \proj^{d^n  + e^n} \]
    is well-defined, then it is given by a linear system of degree
    \[ 2(d^n+e^n - 1)\frac{d^n - e^n}{d-e}. \]
    Moreover, in this case, we can define the multiplier map of the periodic orbits of period $n$,
    \[ \orbmult_{n,(d,e)} : \ccor_{d,e} \ratmap \proj^{ ( \nu_n(d) + \nu_n (e) ) / n} \]
    and it is given by a linear system of degree at most
    \[ \frac{2 ((d^n-1)\nu_n(d) - (e^n-1)\nu_n (e))}{n (d - e) } .\]
\end{prop}
\begin{proof}
    By \cite[Section 7]{gotou2023dyn_corr_moduli} and the Sylvester formula, the degree of fixed point multiplier map is $2(d+e-1)$. Since the morphism $\Psi_n : \ccor_{d,e} \ratmap \ccor_{d^n , e^n}$ is given by a linear system of degree $\frac{d^n - e^n}{d-e}$, we obtain the first assertion.
    Moreover, the morphism $\lambda_n := \lambda_{n,(d,e)}$ is given by 
    \begin{align}
        \lambda_n([f]) & = \lambda_{1,(d,e)} ([\Psi_nf]) \nonumber \\
        & = \left[ \prod_{z : \Psi_n f(z,z) = 0}\left( \partial_x\Psi_n f(z,z) dx + \partial_y\Psi_nf(z,z) dy \right) \right] \in \proj (D_N), \label{nmultprod}
    \end{align}
    where $N = \deg_z \Psi_n f(z,z) = d^n+e^n$.
    The well-definedness of $\lambda_n$ implies that the factors of \eqref{nmultprod} are not zero. Therefore, we can define $\ptmult_n([f])$ as
    \begin{equation}
         \ptmult_n([f]) := \left[ \prod_{z : \Pi_n^* f(z) = 0}\left( \partial_x\Psi_n f(z,z) dx + \partial_y\Psi_nf(z,z) dy \right) \right] \in \proj (D_M), \label{nmultptdef}
    \end{equation}
    where $M = \deg_z  \Pi_n^* f(z) = \nu_n(d) + \nu_n(e)$.
    By \cite[Remark 7.6]{gotou2023dyn_corr_moduli}, we can write this multiplier map as
    \begin{align}
        \ptmult_n ([f]) & = \left[ \fct{res}_z\left( \Pi_n^*f(z), \partial_x\Psi_n f(z,z) dx + \partial_y\Psi_nf(z,z) dy \right) \right] \nonumber \\
        & = \left[ \fct{res}_z\left( \Pi_n^*f(z), d_z \Psi_nf(z,z) dz_0 + z_0z_1\Omega^1\Psi_nf(z) dz_1 \right) /A_{n,0}A_{n,1}\right], \label{nmultptomega}
    \end{align}
    where $A_{n,0}$ and $A_{n,1}$ are the coefficients of respectively $z_0^M$ and $z_1^M$ of $\Pi_n^*f(z)$.
    Therefore, by the Sylvester formula, the rational map $\lambda^\circ_n([f])$ is given by the linear system given by the coefficients of $dz_0^idz_1^{M-i}$ of
    \begin{equation}
        \fct{res}_z\left( \Pi_n^*f(z), d_z \Psi_nf(z,z) dz_0 + z_0z_1\Omega^1\Psi_nf(z) dz_1 \right) /A_{n,0}A_{n,1}, \label{nmultform}
    \end{equation}
    and their degree is
    \begin{equation}
        (d^n+e^n) \frac{\nu_n(d) - \nu_n (e)}{d - e} +  (\nu_n (d) + \nu_n (e)) \frac{d^n - e^n}{d - e} - 2\frac{\nu_n(d) - \nu_n (e)}{d - e}. \label{nmulptdeg}
    \end{equation}
    Any periodic orbit of formal period $n$, of a correspondence defined by $f(x,y)$ is given by a tuple of points $(z_0, z_1,\ldots , z_n = z_0)$ such that $f(z_i , z_{i+1} ) = 0$. From the differential of composite functions, for the periodic points of the same periodic orbits, the factor in \eqref{nmultptdef} takes the same value, that is,
    \begin{equation}
         \ptmult_n([f]) = \left[ \prod_{\substack{ (z_0,\ldots , z_{n-1}) : \\ \text{Periodic orbits of }f(x,y) }}\left( dx + \frac{ \partial_y\Psi_nf(z_0,z_0)}{\partial_x\Psi_n f(z_0,z_0)} dy \right)^n \right]. \label{nmultpow}
    \end{equation}This leads that the map $\ptmult_{n,(d,e)}([f])$ is given by an $n$-th power of some rational function. Therefore, we can define $\orbmult_n$ as an $n$-th root of some quotient of \eqref{nmultform}.
\end{proof}
\begin{rem}\label{immults}
    By definition, for the Veronese embedding \[ \nu_n : \proj^{M/n} \simeq \proj (k[x,y]_{M/n} ) \to \proj( k[x,y]_{M}) \simeq \proj^{M} : f\mapsto f^n, \] we have $\ptmult_n = \nu_n \circ \orbmult_n$. Moreover, since for any periodic orbit $(z_0, \ldots , z_{n-1},z_n = z_0)$ we have
    \[ \frac{ \partial_y\Psi_{mn}f(z_0,z_0)}{\partial_x\Psi_{mn} f(z_0,z_0)} = \left( \frac{ \partial_y\Psi_nf(z_0,z_0)}{\partial_x\Psi_n f(z_0,z_0)} \right)^m, \]
    we can see that 
    \[ \fct{Im} \Lambda_n \simeq \fct{Im} \Lambda^{\bullet}_{n} \simeq \fct{Im} \Lambda^{\circ}_{n} , \]
    where
    \[ \Lambda^\alpha_{n} := \prod_{m : m|n} \lambda^\alpha_{m} \text{ for } \alpha \in \{ \circ, \bullet \}. \]
\end{rem}
\begin{rem} Despite the form in \eqref{nmultptomega} is given by an $n$-th power of some polynomial, it is difficult to obtain more explicit form of the $n$-th root $\orbmult_n$. This phenomenon happens in computing resultant by Cayley's formula (\cite{eisenbud2003resultants}). If $n = 2$, Pfaffian is sometimes used to compute Cayley's formula. Whether analogous method exist for the second multiplier map is a problem. As we see in \eqref{sgncpt}, if we have a method to choose a specified branch of the roots, we can compute the root directly by interpolation.
\end{rem}

\subsection{Evaluation}
We use the following result to calculate the volume.
\begin{thm}\label{chhs} (\cite{dc-h-h-s2020hilbert}) Let $V \in \fct{Rep}_k(V_1)$ be a representation of $\slt (k)$ and $\fct{dim}_k V = n$. Then the Taylor expansion of the Hilbert function of the invariant ring $I(V) := k[V]^{\slt (k)}$ at $t=1$ has the form
\[ H_{I(V)}(t) = (1-t)^{-d+3} \cdot \sum_{i = 0}^{\infty} \gamma_i (1-t)^i. \]
Let $(a_i)$ be the positive weights of $V$ and $l$ be the length of the sequence. Then we have
\[ \gamma_0 = \gcd( 2, a_1 , \ldots , a_l ) \frac{s_{(l-3,l-3,l-3 , l-4 , l-5, \ldots ,2, 1,0)} (a_1 , a_2, \ldots , a_l) }{s_{(l-1,l-2,l-3, l-4 , l -5, \ldots , 2,1,0 )} (a_1, a_2 , \ldots , a_l) }. \]
\end{thm}
\begin{rem}
In \cite{dc-h-h-s2020hilbert}, higher terms ($\gamma_1, \gamma_2, \gamma_3$) are also computed. 
\end{rem}

Throughout this subsection, we put $n := d+e$. We give a rough estimate of $\gamma_0$ for the case $V = V_d \otimes V_e$ of the moduli space of dynamical systems.

\begin{lemm}\label{schureval} We have 
\[ \fct{Vol} I(V_{d} \otimes V_{e}) \leq \frac{\gcd(n,2)}{2(n-2)(n-1)n}.  \]
\end{lemm}
\begin{proof}
We put $\alpha := (l-3,l-3,l-3,l-4,l-5, \ldots , 1)$ and $\delta := (l-1, l-2 , \ldots , 2, 1)$. Let $\Phi_k : \fct{SST}_l (\alpha ) \to \fct{SST}_l (\delta )$ for $k = 1,2$ be the map such that for any $f \in \fct{SST}_{l}(\alpha )$
\[ \Phi_k(f) (i,j) := \begin{cases} f(i,j) & ((i,j) \in T(\alpha)) \\ l-1 & ((i,j) = (1,l-2)) \\ l & ((i,j)=(2,l-2)) \\ l-1 & ((i,j) = (1,l-1) ,\ k=1) \\ l & ((i,j)=(1,l-1),\ k=2) \end{cases}. \]
The maps $\Phi_1$ and $\Phi_2$ are both injective and the images are disjoint. 
Therefore by Theorem \ref{kostka} we have
\begin{align} s_{\delta}(x_1 ,\ldots , x_l) =& (x_{l-1} x_l^2 + x_{l-1}^2 x_l) s_{\alpha}(x_1 , \ldots , x_l) \nonumber\\ 
& + (\text{polynomial with nonnegative coefficients}). \label{schqr} \end{align}
 The three largest among the weights of the representation $V_d \otimes V_e$ are $(n-2,n-2,n)$ and other weights are smaller than $n-2$.
By substituting the weights into (\ref{schqr}), we obtain
\[ s_\delta ( \nu , n-2 , n-2 , n) \geq 2(n-2)(n-1)n \cdot s_\alpha ( \nu ,n-2,n-2,n), \]
where we put the sequence of positive weights smaller than $n-2$ by $\nu$.
By Theorem \ref{chhs}, we have 
\begin{align*} \fct{Vol}(I(V_d \otimes V_e)) &= \gcd(2,n)  \frac{s_\alpha ( \nu , n-2,n-2,n)}{ s_\delta ( \nu , n-2,n-2,n)} \\
&\leq \gcd(2,n) \frac{1}{2(n-2)(n-1)n}.
\end{align*}
\end{proof}

\begin{thm}
Let $p$ be a prime number. If the first and the $p$-th multiplier map to the image
\[\Lambda^\circ_p := \orbmult_{1,(d,e)} \times \orbmult_{p,(d,e)} : \dyn_{d,e} \ratmap \Lambda (\dyn_{d,e}) \subset \proj (D_{d+e}) \times \proj (D_{M}) \]
is finite, then its degree is at most \[ \frac{\gcd(n,2) N^{de+n-3} (n-3)! (de-3)! }{2n \cdot (de+n-3)!}, \]
where
\[ N := 2(d+e-1) + \frac{2( (d^p-1)(d^p-d)-(e^p-1)(e^p-e))}{p(d-e)}. \]
\end{thm}
\begin{proof}
We put $A := I(V_d \otimes V_e)$. Let $L_1$ and $L_p$ be the linear systems which gives $\orbmult_{1}$ and $\orbmult_p$. By Proposition \ref{linsys}, we can take the linear systems in $A$ such that whose degrees are respectively at most 
\[ 2(d+e-1) \text{ and } \frac{2( (d^p-1)(d^p-d)-(e^p-1)(e^p-e))}{p(d-e)}. \]
By assumption and algebraic independence of discriminant-resultant \cite{gorbovickis2015rat_mult_alg_indep}, we can take $n$ elements $f_1 , \ldots , f_n$ in $L_1$ and $\dim \dyn_{d,e} - (n-1) = de-2$ elements $g_1, \ldots , g_{de-2}$ in $L_m$ to be algebraically independent.

Let $k[L_1\otimes L_p] \to A$ be the morphism of graded $k$-algebras defined by $L_1 \otimes L_p \ni f \otimes g \mapsto fg \in A$ and $B_{(1,p)}$ be its image. Then we have the degree of the morphism $\Lambda$ is the extension degree $[KP(A) : KP(B_{(1,p)} )]$ of rational function fields. 

The graded subalgebra $B_{(1,p)} \supset B := k[f_ig_j \mid 1 \leq i \leq n,\ 1 \leq j \leq de-2]$ of $A$ has the Hilbert series
\[ H_B(t) = \sum_{i = 0}^{\infty} \binom{i+n-1}{n-1} \binom{i+de-3}{de-3} t^{iN}. \]
By Corollary \ref{dvcor}, we have
\begin{align*}
\fct{Vol}(B) = \frac{1}{N^{de+n-3}} \cdot \frac{(de+n-3)!}{(de-3)! (n-1)!}.
\end{align*}
Therefore we have
\begin{align} 
[KP(A) : KP(B_{(1,p)})] &\leq [KP(A) : KP(B)] \nonumber \\
& \leq \frac{\fct{Vol}(A)}{ \fct{Vol}(B)} \nonumber \\
& \leq \frac{\gcd(n,2)}{2n(n-1)(n-2)} \cdot \frac{N^{de+n-3} (n-1)! (de-3)! }{(de+n-3)!} \label{eval} \\
&=\frac{\gcd(n,2) N^{de+n-3} (n-3)! (de-3)! }{2n \cdot (de+n-3)!} \nonumber
\end{align}
from Proposition \ref{Voldeg}.
\end{proof}
\begin{rem}
    By skipping Lemma \ref{schureval}, we can use 
    \begin{equation}
        \frac{s_{\delta}(\text{positive weights of }V_d\otimes V_e)}{s_{\alpha}(\text{positive weights of }V_d\otimes V_e)} \label{svdve}
    \end{equation} 
    instead of $1/2n(n-1)(n-2)$ in \eqref{eval}. Experimentally \eqref{svdve} looks like of order $O(n^{-(4+O(1))})$, but this difference of orders may be very small comparing to $N^{de+d+e-3} $.
\end{rem}
\begin{rem}\label{dimsnd}
From number of periodic orbit, and Holomorphic Lefschetz formula (\cite{illusie1977sga5}, \cite{gotou2023dyn_corr_moduli}) for $C$ and $\Psi_2 C$, the dimension of fiber of $\Lambda_2^\circ$ is at least
\begin{align}
& (d+e-1) + \frac{(d^2-d) + (e^2-e)}{2}-1 - (de+d+e-3) \nonumber \\ 
& = \frac{(d-e)^2 -(d+e) + 2}{2}.
\end{align}
In particular, the map $\Lambda_2^\circ$ can be generically finite to its image only if $(d-e)^2 \geq d+e-2$.
\end{rem}

\section{Birationality of the Second Multiplier Map of Cubic Maps} \label{sc:cubbirat}
In this section, we give two proofs of the following theorem.
\begin{thm}\label{cubbirat} The multiplier map $\Lambda_{2,(1,3)}$ is birational to its image. 
\end{thm}

\subsection{Finite field reduction}\label{ss_ffr}
The counting of the degree of
\[\Lambda_{2} := \lambda_{1,(1,3)} \times \lambda_{2,(1,3)} : \dyn_{1,3} \to \proj^{3} \times \proj^{9} \]
to its image is done in \cite{hutz-tepper2013multiplier} by the following method. First, we fix a point $P \in \lambda_1(\dyn_{1,3})$ and consider the inverse image $l := \lambda_{1}^{-1} (P)$. An explicit morphism $\phi_P : \pone \to \ccor_{1,3}$ such that $l$ is birational to the image of $\pone \xrightarrow{\phi_P } \ccor_{1,3} \xrightarrow{\pi } \dyn_{1,3}$ is given in \cite{hutz-tepper2013multiplier}. We denote the endomorphism on $\pone$ indicated by the point $\phi_P (a)$ by $\phi_{P,a}$. Then we will solve the equations in the two variables $a$ and $b$,
\begin{align}
\phi_{P,a}^2(b) & = b \label{sndper} \\
(\phi_{P,a}^2 )'(b) & = \lambda \label{sndmult} \\
\phi_{P,a}(b) & \neq b \label{notfixed}
\end{align}
for a given $P$ and $\lambda$. The equations (\ref{sndper}) and (\ref{sndmult}) are of degree 9 and 16 respectively, in variables $a$ and $b$.
By a MAGMA computation over a finite field, we obtain the solutions as a 0-dimensional closed subscheme $Z$ of degree 144 on $\proj^2$. Under a base-change to the algebraically closed field, the support of $Z$ consists of 18 points. Six of them are non-reduced and does not satisfy the inequality (\ref{notfixed}). Remaining 12 points satisfies (\ref{notfixed}), moreover the MAGMA computation shows that they are reduced.

That was the computation done in \cite{hutz-tepper2013multiplier}. We proceed computation from here. At first, we note that for a solution $(a,b)$ of \eqref{sndper},\eqref{sndmult} and \eqref{notfixed}, the points $(a , \phi_{P,a}(b))$ is also a solution of equation. Therefore, we obtain 6 rational maps $\phi_{P,a}$ with periodic points of period two. For a value $\lambda$, the solutions are given by explicit values of $(a,b)$. For the 6 rational maps, we compute other multipliers of periodic points of period 2. Then we obtain that the values of other multipliers are mutually different, so we obtain that $\Lambda_{3,2}$ is injective.

\subsection{Direct computation with Graded-decomposition and interpolation}\label{ss_dc}
In this section, we show Theorem \ref{cubbirat} by computing the full formula of the second multiplier map $\Lambda^\circ_{2,(1,3)}$.

In the computational process of the explicit expression, we use the information that the resulting polynomials are $\slt$-invariant. Our algorithm (Subsection \ref{ss_intp_alg}) of graded-piece-wise computation is applied for limited cases, but this method makes the computation much faster.

We refer the expression of generators of the invariant ring $A := I(V_1 \otimes V_3) = I(V_4 \oplus V_2)$ given in \cite{west2015moduli}. The invariant ring is given by 
\[ A \simeq k[d,i,j,a,b,c] / r, \]
where $d,i,j,a,b,c$ are generators of degree respectively $2,2,3,3,4,6$ and $r$ is the relation 
\[ 2 c^2=\frac{1}{54} d^3 i^3-\frac{1}{9} d^3 j^2-\frac{1}{12} d i^2 a^2-\frac{1}{3} j a^3+d j a b+\frac{1}{2} i a^2 b-\frac{1}{2} d i b^2-b^3. \]
Let $f_4(z)$ and $f_2(z)$ be the fundamental covariants of $V_4$ and $V_2$ respectively (denoted by $\mathbf{f}$ and $\mathbf{g}$ in \cite{west2015moduli} respectively). By \cite[Remark 7.6]{gotou2023dyn_corr_moduli}, the first multiplier map $\lambda_1$ is given by the linear system consists of discriminant-resultants (named in \cite{gotou2022brac_disc_res})
\[ \sigma_r := DR_{4,r}(f_4,f_2)\ (r = 0,2,3,4) \text{ of degree }(6-r,r), \]
which are defined by
\[ \sum_{r = 0}^4 DR_{4,r}(f_4,f_2) t^r = \fct{res}_{z}(f_4, \partial_zf_4 + zf_2t). \]
We put
\[ \Sigma_{\pm} := \sigma_0 + \sigma_2 \pm \sigma_3 + \sigma_4. \]
By Proposition \ref{linsys}, we can have a linear system of $\lambda_2^\circ$ of degree at most 24. Let 
\[ L_2 (t) := \fct{res}_{z}(\Pi_2^*f(z), d_z \Omega^0(\Psi_2 f) (z) + t \cdot \Omega^1 (\Psi_2f) (z)). \]
In this case, $L_2(t)$ has a divisor $\Sigma_-^4$ and we can take a square root of $L_2(t)/\Sigma_-^4$. We put the square root of $L_2(t)$ as
\[ \sqrt{L_2(t)} =: \delta (t) =: \delta_0 + \delta_1t + \delta_2t^2 + \delta_3t^3. \]
The forms $\delta_i$'s are invariants in $A$ of degree $(48-4\cdot 6)/2 = 12$. We have
\begin{align*}
    \delta_0^2 \cdot \Sigma_-^4 & = \fct{res}_z(\Pi_2^*f(z) , d_z \Omega^0(\Psi_2f)(z)) \\
    & =\Delta_z(\Pi_2^*f(z)) \cdot \fct{res}_z(\Pi_2^*f(z), \Pi_1^*f(z))
\end{align*}
and by a direct computation we obtain
\begin{align*}
    \Delta_z(\Pi_2^*f(z)) & = \Sigma_+ \cdot \Sigma_-^2 \cdot \phi^2 \text{ and }\\
    \fct{res}_z(\Pi_2^*f(z), \Pi_1^*f(z)) & = \Sigma_+ \cdot \Sigma_-^2 \text{, where}\\
    \phi & =2^{-27}\cdot (d^3 - 12d^2i + 48di^2 - 64i^3 + 384j^2 + 288ja + 54a^2), 
\end{align*}
so we have $\delta_0 = \pm \Sigma_+ \cdot \phi$. By fixing the sign to be $+$, we can compute $\delta(t)$ and obtain 
\begin{align}
    \delta_1  = -\frac{1}{2}  (-9\sigma_0 + \sigma_2 + 6\sigma_3 + 11\sigma_4 )\cdot \phi, 
    \delta_2, \delta_3  \in A_{12}. \label{sgncpt}
\end{align}

Here we remark that
\[ K_2 := K(\Lambda_2(\dyn_{1,3})) = k\left( \frac{\sigma_i}{\sigma_j} , \frac{\delta_i}{\delta_j} \right) \subset KP(A). \]
So we start from the algebra
\[ B_1 := k[\sigma_i \cdot \phi (i = 0,2,3,4), \delta_2 ,\delta_3 ] (\subset A)\]
to larger sub-graded-algebras of $A$ with keeping the condition $KP(B_i) = K_2$.
By seeking factorizable linear combinations of the generators of $B_1$, we find
\begin{align*}
& \delta_2 + \phi (-10 \sigma_0 + \sigma_2 + 10 \sigma_3 ) = \Sigma_- \cdot \psi, \\
&\text{where }\psi = \frac{1}{2^{31}}(-26048i^3 + 9936i^2d - 884id^2 + 7d^3  \\
& \hspace{80pt}- 102912j^2 - 38784ja - 72a^2 - 9600ib + 2400db)\text{ and} \\
& (\delta_2 + \delta_3) - \frac{1}{2}(11\sigma_0 + \sigma_2 - 4 \sigma_3 - 9 \sigma_4)\phi
=\frac{5^3}{2^{19}} (\Sigma_- + \sqrt{2} \sigma_3)(\Sigma_- - \sqrt{2} \sigma_3).
\end{align*}
Therefore, we can replace $B_1$ by 
\[ B_2 := k[ \sigma_i, \phi, \psi ]. \]
By computing the elimination ideal of generators of $B_2$, we obtain that the only relation among the generators is only a relation of degree $60$ (degree $10$ polynomial of $\sigma_i,\phi,\psi$), so we can see that the Hilbert series of $B_2$ is given by
\[ H_{B_2}(t) = \frac{1-t^{60}}{(1-t^6)^6}. \]
Here we have
\begin{align*}
 H_A(t) & = \frac{1+t^6}{(1-t^2)^2(1-t^3)^2 (1-t^4)}.
\end{align*}
We recall that for any graded algebra $C$, $C^{[n]} := \bigoplus_{i \geq 0} C_{in}$ with $\deg C_{in} = i$.
By a direct computation, we have
\begin{align*}
 H_{A^{[6]}}(t) & = \frac{(1+t)(1+5t + 9t^2 + 4t^3)}{(1-t)^4(1-t^2)} \text{ and }H_{B_2^{[6]}}(t) = \frac{1-t^{10}}{(1-t)^6}.
\end{align*}
By Proposition \ref{Voldeg}, we have 
\[ \deg (\Lambda_2 ) = [KP(A^{[6]}) : KP( B_2^{[6]} )] \leq \frac{\fct{Vol}(A^{[6]}) }{\fct{Vol}(B_2^{[6]})} = \frac{9}{5}, \]
this shows that $\deg (\Lambda_2) = 1$.

\begin{rem}\label{simprel}
    Throughout this ad hoc proof, there are three steps completed unexpectedly easily. The first is that there are factorizable linear combinations including $\delta_2$ and $\delta_3$. The second is that a linear combination moreover belongs to $k[\sigma_i ]_{12}$. The third is that the relation among $\sigma_i,\phi , \psi$ was of degree 60. Because of this small degree (the expected degree from the Hilbert series is 108), we can obtain the result in a few minutes by simply computing the elimination ideal. Moreover, this degree is also the lower bound to determine the extension degree to be $1$.
\end{rem}
\appendix

\section[A]{Appendix: Programs}
\subsection{Programs used in Subsection \ref{ss_ffr}} \label{ss_ffr_pr}
The MAGMA program run in \cite{hutz-tepper2013multiplier} were the following.
\begin{itemize}\item[{}]\texttt{l0:=3;\\
l1:=2;\\
l8:=4;\\
lB:=-5;\\
R<a,B,z>:=ProjectiveSpace(GF(101),2);\\
function h(P,d)
Q:=0;\\
for i:=0 to d do
for j:=0 to d-i do
Q:=Q+ Term(Term(P,a,i),B,j)*z\textasciicircum  (d-i-j);\\
end for;\\
end for;\\
return(Q);\\
end function;\\
function f(x,y)
return((((l0 - 1)*l1 + (-l0 + 1))*x\textasciicircum 3 + ((a*l0*l1 + (-l0 + (-a + 1)))*l8 +
(((-a - 1)*l0 +1)*l1 + (2*l0 + (a - 2))))*x\textasciicircum 2*y + ((-a*l0*l1 + a*l0)*l8 +
(a*l0*l1 - a*l0))*x*y\textasciicircum 2));\\
end function;\\
function g(x,y)
return((((l0 - 1)*l1 + (-l0 + 1))*l8*x\textasciicircum 2*y + (((-l0 + (a + 1))*l1 +
(a*l0 - 2*a))*l8 + (-a*l1 + ((-a + 1)*l0 +(2*a - 1))))*x*y\textasciicircum 2 +
((-a*l1 + a)*l8 + (a*l1 - a))*y\textasciicircum 3));\\
end function;\\
f1:=f(f(B,1),g(B,1));\\
g1:=g(f(B,1),g(B,1));\\
F1:=f1-B*g1;\\
F2:=g1*Derivative(f1,B) - f1*Derivative(g1,B) - lB*g1*g1;\\
G1:=h(F1,9);\\
G2:=h(F2,16);\\
C:=Scheme(R,[G1,G2]);\\
D:=ReducedSubscheme(C);\\
Degree(D);}
\end{itemize}
After running this computation, we firstly compute the coordinates of the points of $ D$.
\begin{itemize}
\item[{}]\texttt{PointsOverSplittingField(D); \\
Output: \\
\{@ (0 : 0 : 1), (0 : 4 : 1), (1 : 1 : 1), (1 : 49 : 1), (4 : 93*\$.1\textasciicircum 7 + 44*\$.1\textasciicircum 6
    + 24*\$.1\textasciicircum 5 + 23*\$.1\textasciicircum 4 + 48*\$.1\textasciicircum 3 + 26*\$.1\textasciicircum 2 + 65*\$.1 + 90 : 1), (4 : 8*\$.1\textasciicircum 7
    + 57*\$.1\textasciicircum 6 + 77*\$.1\textasciicircum 5 + 78*\$.1\textasciicircum 4 + 53*\$.1\textasciicircum 3 + 75*\$.1\textasciicircum 2 + 36*\$.1 + 79 : 1), 
(96 : 27 : 1), (96 : 6 : 1), (18*\$.1\textasciicircum 7 + 50*\$.1\textasciicircum 6 + 68*\$.1\textasciicircum 5 + 24*\$.1\textasciicircum 4 + 
    59*\$.1\textasciicircum 3 + 22*\$.1\textasciicircum 2 + 93*\$.1 + 93 : 55*\$.1\textasciicircum 7 + 55*\$.1\textasciicircum 6 + 80*\$.1\textasciicircum 5 + 
    72*\$.1\textasciicircum 4 + 5*\$.1\textasciicircum 3 + 89*\$.1\textasciicircum 2 + 52*\$.1 + 10 : 1), (18*\$.1\textasciicircum 7 + 50*\$.1\textasciicircum 6 + 
    68*\$.1\textasciicircum 5 + 24*\$.1\textasciicircum 4 + 59*\$.1\textasciicircum 3 + 22*\$.1\textasciicircum 2 + 93*\$.1 + 93 : 15*\$.1\textasciicircum 7 + 
    56*\$.1\textasciicircum 6 + 83*\$.1\textasciicircum 5 + 3*\$.1\textasciicircum 4 + 62*\$.1\textasciicircum 3 + 95*\$.1\textasciicircum 2 + 21*\$.1 + 80 : 1), 
(14*\$.1\textasciicircum 7 + 52*\$.1\textasciicircum 6 + 48*\$.1\textasciicircum 5 + 18*\$.1\textasciicircum 4 + 53*\$.1\textasciicircum 3 + 62*\$.1\textasciicircum 2 + 42*\$.1 + 50 :
70*\$.1\textasciicircum 7 + 7*\$.1\textasciicircum 6 + 3*\$.1\textasciicircum 5 + 27*\$.1\textasciicircum 4 + 47*\$.1\textasciicircum 3 + 32*\$.1\textasciicircum 2 + 64*\$.1 + 28 : 
    1), (14*\$.1\textasciicircum 7 + 52*\$.1\textasciicircum 6 + 48*\$.1\textasciicircum 5 + 18*\$.1\textasciicircum 4 + 53*\$.1\textasciicircum 3 + 62*\$.1\textasciicircum 2 + 
    42*\$.1 + 50 : 88*\$.1\textasciicircum 7 + 85*\$.1\textasciicircum 6 + 73*\$.1\textasciicircum 5 + 50*\$.1\textasciicircum 4 + 65*\$.1\textasciicircum 3 + 
    59*\$.1\textasciicircum 2 + 96*\$.1 + 70 : 1), (75*\$.1\textasciicircum 7 + 95*\$.1\textasciicircum 6 + 57*\$.1\textasciicircum 5 + 100*\$.1\textasciicircum 4 + 
    90*\$.1\textasciicircum 3 + 4*\$.1\textasciicircum 2 + 73*\$.1 + 55 : 5*\$.1\textasciicircum 7 + 72*\$.1\textasciicircum 6 + 70*\$.1\textasciicircum 5 + 39*\$.1\textasciicircum 4 
    + 32*\$.1\textasciicircum 3 + 31*\$.1\textasciicircum 2 + 74*\$.1 + 26 : 1), (75*\$.1\textasciicircum 7 + 95*\$.1\textasciicircum 6 + 57*\$.1\textasciicircum 5 + 
    100*\$.1\textasciicircum 4 + 90*\$.1\textasciicircum 3 + 4*\$.1\textasciicircum 2 + 73*\$.1 + 55 : 86*\$.1\textasciicircum 7 + 93*\$.1\textasciicircum 6 + 
    92*\$.1\textasciicircum 5 + 67*\$.1\textasciicircum 4 + 46*\$.1\textasciicircum 3 + 95*\$.1\textasciicircum 2 + 22*\$.1 + 78 : 1), (95*\$.1\textasciicircum 7 + 
    5*\$.1\textasciicircum 6 + 29*\$.1\textasciicircum 5 + 60*\$.1\textasciicircum 4 + 13*\$.1\textasciicircum 2 + 95*\$.1 + 87 : 66*\$.1\textasciicircum 7 + 33*\$.1\textasciicircum 6
    + 62*\$.1\textasciicircum 5 + 50*\$.1\textasciicircum 4 + 83*\$.1\textasciicircum 3 + 87*\$.1\textasciicircum 2 + 29*\$.1 + 25 : 1), (95*\$.1\textasciicircum 7 + 
    5*\$.1\textasciicircum 6 + 29*\$.1\textasciicircum 5 + 60*\$.1\textasciicircum 4 + 13*\$.1\textasciicircum 2 + 95*\$.1 + 87 : 19*\$.1\textasciicircum 7 + 3*\$.1\textasciicircum 6 
    + 42*\$.1\textasciicircum 5 + 96*\$.1\textasciicircum 4 + 64*\$.1\textasciicircum 3 + 17*\$.1\textasciicircum 2 + 46*\$.1 + 2 : 1), (47 : 1 : 0),
(1 : 0 : 0) @\} \\
Finite field of size 101\textasciicircum 8 }
\end{itemize}
These are the coordinates $(a:b:z)$ of the solutions of \eqref{sndmult} and \eqref{sndper} on $\proj^2$, with the homogenizing variable $z$. The parameters $a$ of 12 reduced points, consisted of 6 values as expected are the following.
\begin{itemize}\item[{}]\texttt{\{4,96,18*\$.1\textasciicircum7 + 50*\$.1\textasciicircum6 + 68*\$.1\textasciicircum5 + 24*\$.1\textasciicircum4 + 59*\$.1\textasciicircum3 + 22*\$.1\textasciicircum2 + 93*\$.1 + 93,
14*\$.1\textasciicircum7 + 52*\$.1\textasciicircum6 + 48*\$.1\textasciicircum5 + 18*\$.1\textasciicircum4 + 53*\$.1\textasciicircum3 + 62*\$.1\textasciicircum2 + 42*\$.1 + 50,
75*\$.1\textasciicircum7 + 95*\$.1\textasciicircum6 + 57*\$.1\textasciicircum5 + 100*\$.1\textasciicircum4 + 90*\$.1\textasciicircum3 + 4*\$.1\textasciicircum2 + 73*\$.1 + 55,
95*\$.1\textasciicircum7 + 5*\$.1\textasciicircum6 + 29*\$.1\textasciicircum5 + 60*\$.1\textasciicircum4 + 13*\$.1\textasciicircum2 + 95*\$.1 + 87\},
}
\end{itemize}
The multipliers are given by:
\begin{itemize}
\item[{}] \texttt{l0:=3;\\
l1:=2;\\
l8:=4;\\
lB:=-5;\\
\\
F<w>:= GF(101,8);\\
R<a,B,z>:=PolynomialRing(F,3);\\
\\
function f(x,y)\\
return((((l0 - 1)*l1 + (-l0 + 1))*x\textasciicircum 3 + ((a*l0*l1 + (-l0 + (-a + 1)))*l8 +
(((-a - 1)*l0 +1)*l1 + (2*l0 + (a - 2))))*x\textasciicircum 2*y + ((-a*l0*l1 + a*l0)*l8 +
(a*l0*l1 - a*l0))*x*y\textasciicircum 2));\\
end function;\\
function g(x,y)\\
return((((l0 - 1)*l1 + (-l0 + 1))*l8*x\textasciicircum 2*y + (((-l0 + (a + 1))*l1 +
(a*l0 - 2*a))*l8 + (-a*l1 + ((-a + 1)*l0 +(2*a - 1))))*x*y\textasciicircum 2 +
((-a*l1 + a)*l8 + (a*l1 - a))*y\textasciicircum 3));\\
end function;\\
f2:=f(f(B,1),g(B,1));\\
g2:=g(f(B,1),g(B,1));\\
F2:=f2-B*g2;\\
redF2:=R!(F2/(f(B,1) - B * g(B,1)));\\
dF2:=g2*Derivative(f2,B) - f2*Derivative(g2,B) - z*g2*g2;\\
\\                                        
function mult(c);\\
return(Resultant(Evaluate(redF2,a,c),Evaluate(dF2,a,c),B));\\
end function;\\        
\\
result:\\
> Factorization(mult(4));\\
\symbol{91}
    <z + 5, 2>,
    <z + 50, 2>,
    <z + 90, 2>
]\\
> Factorization(mult(96));\\
\symbol{91}
    <z + 5, 2>,
    <z + 26, 2>,
    <z + 66, 2>
]\\
> Factorization(mult(18*w\textasciicircum7 + 50*w\textasciicircum6 + 68*w\textasciicircum5 + 24*w\textasciicircum4 + 59*w\textasciicircum3 + 22*w\textasciicircum2 + 93*w + 93));\\
\symbol{91}
    <z + 5, 2>,
    <z + 78*w\textasciicircum7 + 93*w\textasciicircum6 + 47*w\textasciicircum5 + 57*w\textasciicircum4 + 23*w\textasciicircum3 + 80*w\textasciicircum2 + 73*w + 52, 2>,
    <z + 70*w\textasciicircum7 + 79*w\textasciicircum6 + 53*w\textasciicircum5 + 6*w\textasciicircum4 + 42*w\textasciicircum3 + 86*w\textasciicircum2 + 96*w + 53, 2>
]\\
> Factorization(mult(14*w\textasciicircum7 + 52*w\textasciicircum6 + 48*w\textasciicircum5 + 18*w\textasciicircum4 + 53*w\textasciicircum3 + 62*w\textasciicircum2 + 42*w + 50));\\ 
\symbol{91}
    <z + 5, 2>,
    <z + 60*w\textasciicircum7 + 98*w\textasciicircum6 + 74*w\textasciicircum4 + 4*w\textasciicircum3 + 24*w\textasciicircum2 + 24*w + 21, 2>,
    <z + 27*w\textasciicircum7 + 14*w\textasciicircum6 + 59*w\textasciicircum5 + 73*w\textasciicircum4 + 55*w\textasciicircum3 + 12*w\textasciicircum2 + 37*w + 11, 2>
]\\
> Factorization(mult(75*w\textasciicircum7 + 95*w\textasciicircum6 + 57*w\textasciicircum5 + 100*w\textasciicircum4 + 90*w\textasciicircum3 + 4*w\textasciicircum2 + 73*w + 55));\\ 
\symbol{91}
    <z + 5, 2>,
    <z + 100*w\textasciicircum7 + 39*w\textasciicircum6 + 25*w\textasciicircum5 + 12*w\textasciicircum4 + 20*w\textasciicircum3 + 99*w\textasciicircum2 + 21*w + 73, 2>,
    <z + 84*w\textasciicircum7 + 84*w\textasciicircum6 + 91*w\textasciicircum5 + 31*w\textasciicircum4 + 44*w\textasciicircum3 + 70*w\textasciicircum2 + 92*w + 13, 2>
]\\
> Factorization(mult(95*w\textasciicircum7 + 5*w\textasciicircum6 + 29*w\textasciicircum5 + 60*w\textasciicircum4 + 13*w\textasciicircum2 + 95*w + 87));\\           
\symbol{91}
    <z + 5, 2>,
    <z + 89*w\textasciicircum7 + 42*w\textasciicircum6 + 58*w\textasciicircum5 + 91*w\textasciicircum4 + 11*w\textasciicircum3 + 22*w\textasciicircum2 + 91*w + 80, 2>,
    <z + 98*w\textasciicircum7 + 56*w\textasciicircum6 + 71*w\textasciicircum5 + 60*w\textasciicircum4 + 3*w\textasciicircum3 + 11*w\textasciicircum2 + 71*w + 81, 2>
]
}
\end{itemize}
This computation shows the other multipliers of period two orbits are mutually different for the six solutions of (\ref{sndper}).

\subsection{An algorithm for Subsection \ref{ss_dc}}\label{ss_intp_alg}

In order to make up block-decomposed interpolation matrix, we used the following algorithm. The program file written by SAGE\cite{sagemath} is attached, or at \cite{gotou2023prog}.

\begin{algorithm}[H]
\caption{Degree-wise random-sampling interpolation method (probabilistic)}
\begin{algorithmic}
\REQUIRE
\STATE Algorithms to compute $g_1,g_2,\ldots , g_\beta$ and $H = h(g_1 ,\ldots , g_\beta )$,\\ 
\STATE the set of monomials $ M = \left\{ \left. \mathbf{y}^{\mathbf{d}} \right| \mathbf{d} \in \{ \mathbf{d}_1 ,\ldots , \mathbf{d}_l \} \right\} $\\
\STATE such that $h = \sum c_\mathbf{d} \mathbf{y}^{\mathbf{d}}\ (c_\mathbf{d} \in \rat^\beta )$
\STATE A map $\sigma : [n] \to [m]$ such that $g_i(a_1x_{\sigma(1)},\ldots , a_nx_{\sigma(n)})$ is a monomial with coefficient for each $g_i$ and $\mathbf{a} \in \rat^n$.\\

\ENSURE[Probably] The polynomial $h(y_1,\ldots , y_\beta )$
\STATE Separate $M$ by the degree of $\left( \mathbf{g}(\mathbf{x_\sigma} )\right)^{\mathbf{d}}$ into $M_1,\ldots, M_p$
\STATE $l_p := \max \card M_i$
\FOR{ $j$ from 1 to $l_p + l'$}
\STATE Take a random vector $\mathbf{a}_j \in \rat^n$
\STATE Compute $H(\mathbf{a}_j\mathbf{x_\sigma})$
\STATE Compute $g_i(\mathbf{a}_j )$'s and $\mathbf{g}(\mathbf{a}_j)^{\mathbf{d}}$
\ENDFOR
\FOR{ $k$ from 1 to $p$}
\STATE Let $H_{j,k}$ be the coefficient of the term of degree $\mathbf{g}(\mathbf{x}_\sigma)^{\mathbf{d}}$ of $H(\mathbf{a}_j\mathbf{x}_\sigma )$ for $\mathbf{d}$ in $M_k$
\STATE Solve the system of linear equations $H_{j,k} = \sum_{\mathbf{d} \in M_k} c_\mathbf{d} \mathbf{g}(\mathbf{a}_j)^{\mathbf{d}}\ (j = 1,\ldots , l_p+l')$ for $c_{\mathbf{d}}$'s.
\ENDFOR
\STATE $h(\mathbf{y}) = \sum_{\mathbf{d} \in M}c_{\mathbf{d}}\mathbf{y}^{\mathbf{d}}$.
\end{algorithmic}
\end{algorithm}
In our case, we only use addition and multiplication of the polynomials degree less than $H$ to compute $H$, so the numbers of terms appears in the computation are $O(p)$, thus it costs $O(p^2 t_H)$ to compute $H(\mathbf{a}_j\mathbf{x_\sigma})$ par once. Therefore, the computational complexity of we have $O(l_p p^2 t_H + l_p^cp) = O(N(pt_H +l_p^{c-1}))$. In our case $l_p = 70$ and the constants are sufficiently small. Moreover, we set $l' = 5$ in the computation.

\subsection{Data for Subsection \ref{ss_dc}}
Explicit formula of $\delta_2$ and $\delta_3$ are
\begin{align*}
 \delta_2 = & \frac{-1}{2^{31} \cdot 3}(306d^5i - 2072d^4i^2 - 20544d^3i^3 + 300800d^2i^4 - 1691136di^5 \\
& + 3483648i^6 + 96192d^3j^2 - 1286400d^2ij^2 + 10146816di^2j^2 - 16920576i^3j^2 \\
& + 1008d^3ja + 42432d^2ija + 1430784di^2ja - 4451328i^3ja + 459d^3a^2 \\
& + 14316d^2ia^2 - 6768di^2a^2 - 775104i^3a^2 - 1836d^4b + 6624d^3ib \\
& + 173568d^2i^2b - 1850880di^3b + 4672512i^4b - 23887872j^4 - 10616832j^3a \\
& - 156672j^2a^2 + 244224ja^3 + 26136a^4 - 589824dj^2b + 9289728ij^2b \\
& - 672768djab + 4202496ijab - 111744da^2b + 39168ia^2b + 28800d^2b^2 \\
& - 460800dib^2 + 1382400i^2b^2 - 2208d^3c + 103296d^2ic - 1027584di^2c \\
& + 2598912i^3c + 9289728j^2c + 3280896jac - 76032a^2c - 230400dbc + 921600ibc),
\end{align*}
\begin{align*}
\delta_3 = & \frac{1}{2^{31} \cdot 3^3}(1458d^5i + 2904d^4i^2 - 43072d^3i^3 + 2453760d^2i^4 - 11570688di^5 \\
& + 40310784i^6 - 358464d^3j^2 - 10056960d^2ij^2 + 69424128di^2j^2 - 259780608i^3j^2 \\
& - 1296d^3ja - 730944d^2ija + 12379392di^2ja - 58973184i^3ja + 2187d^3a^2 \\
& + 100188d^2ia^2 - 730224di^2a^2 - 1881792i^3a^2 - 8748d^4b - 95328d^3ib \\
& + 2674944d^2i^2b - 20113920di^3b + 82861056i^4b + 107495424j^4 + 17915904j^3a \\
& - 3856896j^2a^2 - 2521728ja^3 + 143748a^4 + 18413568dj^2b - 161243136ij^2b \\
& + 8280576djab - 21399552ijab - 693792da^2b + 5664384ia^2b + 475200d^2b^2 \\
& - 12787200dib^2 + 43545600i^2b^2 + 7776d^3c + 1173888d^2ic - 7921152di^2c \\
& + 49268736i^3c - 6912000b^3 - 71663616j^2c - 3981312jac + 1672704a^2c \\
& - 5529600dbc + 49766400ibc).
\end{align*}

The relation among $\sigma_i, \phi,\psi$ has 1261 terms. The data is attached, or at \cite{gotou2023prog}.

\bibliographystyle{amsalpha}
\bibliography{refs}
\end{document}